\RequirePackage{fix-cm}
\documentclass[smallextended]{svjour3}       

\usepackage{graphicx}
\usepackage{color}
\usepackage{float}
\usepackage{amsmath}
\usepackage{amssymb}
\usepackage{enumitem}
\usepackage{ulem}
\usepackage{amssymb,mathtools}
\usepackage{color}
\makeatletter

\floatstyle{ruled}
\newfloat{algorithm}{tbp}{loa}
\providecommand{\algorithmname}{Algorithm}
\floatname{algorithm}{\protect\algorithmname}

\newtheorem{thm}{\protect\theoremname}

\newtheorem{prop}[thm]{\protect\propositionname}

\newtheorem{defn}[thm]{\protect\definitionname}

\newtheorem{rem}[thm]{\protect\remarkname}

\newtheorem{lem}[thm]{\protect\lemmaname}
\ifx\proof\undefined
\newenvironment{proof}[1][\protect\proofname]{\par
\normalfont\topsep6\p@\@plus6\p@\relax
\trivlist
\itemindent\parindent
\item[\hskip\labelsep\scshape #1]\ignorespaces
}{%
\endtrivlist\@endpefalse
}
\providecommand{\proofname}{Proof}
\fi

\makeatother

\providecommand{\definitionname}{Definition}
\providecommand{\lemmaname}{Lemma}
\providecommand{\propositionname}{Proposition}
\providecommand{\remarkname}{Remark}
\providecommand{\theoremname}{Theorem}

\newcommand{\A}{\mathcal{A}}

\newcommand{\R}{\mathcal{R}}

\newcommand{\E}{\mathcal{E}}

\newcommand{\M}{\mathcal{M}}
\newcommand{\T}{\mathcal{T}}
\newcommand{\SA}{\mathcal{S}}

\newcommand{\UC}{\mathcal{U}}

\usepackage{latexsym}

\begin{document}

\title{Stochastic Approximation on Riemannian manifolds
}

\author{Suhail M. Shah
}

\institute{Suhail M. Shah\at
              Department of Electrical Engineering, Indian Institute of Technology Bombay, Powai, Mumbai 400076, India \\
              \email{suhailshah2005@gmail.com}\\
              \textbf{Acknowledgement :} The author would like to acknowledge the support of Prof. Vivek S. Borkar in preparing this manuscript, in particular for spending a lot of time in pointing out the errors and providing valuable insight and references.
          }

\date{Received: date / Accepted: date}

\maketitle

\begin{abstract}
The standard theory of stochastic approximation (SA) is extended to the case when the constraint set is a Riemannian manifold. Specifically,  the standard ODE method  for analyzing SA schemes  is extended to iterations constrained to stay on a manifold using a retraction mapping. In addition, for submanifolds of a Euclidean space, a framework is developed for a projected SA scheme with approximate retractions. The framework is also extended to non-differentiable constraint sets.

\end{abstract}

\section{Introduction}
\label{intro}
In many situations in engineering and other fields, it is of interest to find the roots of some unknown function $h\,:\,\mathbb{R}^n \to \mathbb{R}^n$ to which we have access to only through noisy measurements. Stochastic Approximation, originally introduced by Robbins and Monro in \cite{Robbins}, provides an incremental scheme for doing this. It involves running the following iteration 
$$x_{n+1} = x_n + a_n (h(x_n) +M_{n+1}), $$
where stepsizes $a(n) > 0$ satisfy the assumptions:
$$\sum_na(n) = \infty, \ \sum_na(n)^2 < \infty,$$
and $\{M(n)\}$ is the noise term, usually assumed to be a martingale difference sequence. A generalizaiton of this is when the relevant function is defined on a manifold. There has been considerable interest in optimization and related algorithms on manifolds, particularly on matrix manifolds \cite{Absil}, \cite{Helmke}. But the corresponding development for such algorithms in presence of noisy measurements has been lacking. A notable exception is \cite{Bonn} which analyzed stochastic gradient schemes on Riemannian manifolds.

The contributions of this work are as follows : First and foremost, we extend the Ordinary Differential Equations (ODE) method for analyzing stochastic approximation schemes (\cite{BorkarBook}, \cite{Kushner}, \cite{Ljung}) to Riemannian manifolds. The usual addition operation of the Euclidean Robbins-Monro scheme is modified using a retraction (see Definition 1,2). We consider a much more general model than \cite{Bonn} and furthermore, our approach is quite distinct from that of \textit{ibid.}, which is inspired from \cite{Bot}. Studying the convergence properties using the associated ODE connects the stochastic approximation theory on manifolds with that of ODEs on manifolds. The upside here is that  ODEs on manifolds are very well studied and have a rich theory (see \cite{Hairer}). This opens up the possibility of adapting analytic and computational techniques in the latter domain for stochastic approximation on manifolds in future.

The second contribution is in the context of projected SA for submanifolds embedded in a Euclidean space as well as for general constraint sets which may not have a manifold structure. In many situations the iterates perforce are executed in an ambient space and a suitable correction at each iterate is required to pull it back to the manifold. For example, for an iteration over the unit sphere, the noisy measurement or even the discretization error may put the next iterate off the sphere and one would typically normalize it to bring it back to the sphere. Such issues also arise, e.g., in gradient projection algorithms. We approach projected SA by taking recourse to ideas developed in \cite{Mathkar} which studied an iteration which used a fixed nonlinear map perturbed by a stochastic approximation like expression with decreasing stepsizes, i.e.,
$$x(n+1) = F(x(n)) + a(n)(h(x(n)) + M(n+1)),$$
$F, h, \{M(n)\}$ satisfy suitable technical hypotheses. The main claim was that asymptotically one recovers the limiting behavior of the projected o.d.e.\
$$\dot{x}(t) = \Gamma(h(x(t))),$$
where $\Gamma$ is the Frechet differnetial of the associated projection map to the set of fixed points of $F$. These results were somewhat restrictive in the regularity assumptions imposed on $F$. This lacuna was worked around in \cite{SuhBor} where $F$ was replaced by a suitable time dependent sequence of maps $\{F_n\}$ in a manner that ensured the above under weaker technical conditions, albeit for the special case when the projection is to the intersection of compact convex sets. In the present work, we develop an analogous scheme wherein the iteration sits in an ambient Euclidean space, but the choice of $\{F_n\}$ ensures the correct asymptotic behavior on the manifold. It is worth noting that both \cite{Mathkar} and \cite{SuhBor} consider distributed algorithms where multiple processors coordinate their computations for a common goal. As this is not our concern here, we downplay this aspect of \cite{Mathkar}.

The rest of the paper is organized as follows. In Section 2 we prove the the main convergence result pertaining to SA on manifolds,  where an ODE approximation is established. Section 3 studies retractions in the framework of submanifolds. We conclude the paper in Section 4 with some comments and future directions.

Before proceeding further we recall some basic facts and definitions about Riemannian geometry which we will use throughout the rest of the paper. Throughout the rest of the paper we let $\M$ denote a connected Riemannian manifold. A smooth $n$-dimensional manifold is a pair $(\M,\A)$, where $\M$ is a Haussdorf second countable topological space and $\A$ is a collection of charts $\{\mathcal{U}_\alpha,\psi_{\alpha}\}$ of the set $\M$, i.e., the following holds :
\begin{enumerate}
\item The collection $\{\mathcal{U}_\alpha\}$ of open sets in $\M$ covers $\M$, i.e.,
$$ \cup_{\alpha} \mathcal{U_\alpha} = \M.$$

\item Each $\psi_{\alpha}$ is a bijection of $\mathcal{U}_{\alpha}$ onto an open set of $\mathbb{R}^n$.

\item for any pairs $\alpha, \beta$ with $\mathcal{U}_{\alpha} \cap \mathcal{U}_{\beta} \neq \emptyset  $, the set $\psi_\alpha(\mathcal{U}_{\alpha} \cap \mathcal{U}_{\beta} )$ and $\psi_\beta(\mathcal{U}_{\alpha} \cap \mathcal{U}_{\beta})$ are open sets in $\mathbb{R}^n$ and the change of coordinates
$$ \psi_{\beta} \circ \psi^{-1}_\alpha : \mathbb{R}^n \to \mathbb{R}^n$$
is smooth on its domain $\psi_\alpha(\mathcal{U}_{\alpha} \cap \mathcal{U}_{\beta} )$.
\end{enumerate}
By a Riemannian manifold we mean a manifold whose tangent spaces are endowed with a smoothly varying inner product $\langle \cdot,\cdot \rangle_x$ called the Riemannian metric.
The tangent space at any point $x \in \M$ is denoted by $\T_x \M$ (for definition see \cite{Absil} or \cite{dcarmo}). We recall that a tangent space admits a structure of a vector space and for a Riemannian manifold it is a normed  vector space. The tangent bundle $\T\M$ is defined to be the disjoint union $\cup_{x \in \M}\{x\}\times\T_x\M$. The normal space at the point $x$ denoted by $\mathcal{N_{M}}(x)$ is the set of all vectors orthogonal (w.r.t to $\langle \cdot,\cdot \rangle_x$) to the tangent space at $x$. Using the norm, one can also define the arc length of a curve $\gamma : [a,b] \to \M$ as
$$L(\gamma)= \int_a^b \sqrt{\langle \dot{\gamma}(t),\dot{\gamma}(t) \rangle_{\gamma(t)}} \, dt.$$

\begin{defn}
A geodesic on $\M$ is a curve that locally minimizes the arc
length (equivalently, these are the curves that satisfy $\gamma''\in\mathcal{N_{M}}(\gamma(t))$
for all $t$). The exponential of a tangent vector $u$ at $x$, denoted
by $\text{exp}_x(u)$, is defined to be $\Gamma(1,x,u)$,
where $t\to\Gamma(t,x,u)$ is the geodesic that satisfies
\[
\Gamma(0,x,u)=x\;\text{ and }\;\frac{d}{dt}\Gamma(0,x,u)\rvert_{t=0}=u.
\]
\end{defn}

We let $d(\cdot,\cdot)$ denote the Riemannian distance  between any two points $x,y\in \M$, i.e., 
$$ d \,:\, \mathcal{M} \times \mathcal{M} \to \mathbb{R}\, : \, d(x,y) = \inf_{\Gamma}L(\gamma),$$
where $\Gamma$ is the set of all curves in $\mathcal{M}$ joining $x$ and $y$. We recall that $d(\cdot,\cdot)$ defines a\textit{ metric }on $\M$. By the neighborhood $\mathcal{U}_{x}$  of a point $x$ we mean the normal neighborhood or the geodesic ball centered at $x$. The coordinate chart for this neighborhood is provided by
$$ \psi_x = E^{-1}\circ \text{exp}^{-1}_x : \mathcal{U}_{x} \to \mathbb{R}^n, $$
where $E:\mathbb{R}^n \to \T_x\M$ is the isomorphism mapping a point $x=(x(1),...,x(n))\in \mathbb{R}^n$ to a point in $\T_x\M$ expressed in the  orthonormal basis $\{E^i\}$ for $\T_x\M$, i.e $E(x)=\sum_i x(i)E^i$. The coordinates of the point $x$ under $\psi_x$ are $0$.

\section{Stochastic Approximation on Manifolds}

Suppose we have a smooth vector field $H\,:\,\mathcal{M} \to \mathcal{T}\mathcal{M} $  assigning to each point $x \in \M$ a tangent vector $H(x) \in \T_x \M$.   We want to find the zeroes of $H(\cdot)$ based on its samples corrupted by a martingale difference noise $M_{x}\in\mathcal{T}_{x}\mathcal{M}$. A common scenario where this happens is when we have to minimize a smooth function :
\[
\min_{x\in\M} \big\{ L(x) \doteq \mathbb{E}_{z}Q(x,z)=\int Q(x,z)dP(z)
 \big\}\]
for some function $Q : \mathcal{M}\times \mathcal{Z} \mapsto \mathbb{R},$ where $x\in\mathcal{M}$ is the minimization
parameter and $P$ is a probability measure on a measurable space
$\mathcal{Z}$. The minimization is to be performed based on noisy measurements of a black box that outputs at time $n$ on input $x_n$, the quantity $Q(x_n, \xi_n)$ where $\xi_n$ is an independent sample with law $P$. In this case one sets $H(x) = E\left[Q(x, \xi)\right]$ and $M_{n+1} = Q(x_n, \xi_n) - H(x_n)$.

The most obvious way to do this would be to take the noisy sample
of the vector field and perform a simple iterative gradient descent like scheme. However, such an iteration involving the addition of two points  would not be well defined for manifolds.  A way to remedy this is to use the notion of  a geodesic
and the associated concept of the exponential map.

The natural generalization of the SA scheme to the case of manifolds
using the exponential map would yield\footnote{For ease of notation we drop hereafter the superscript $x_k$ on the noise term $M_{k+1}$ which is understood to belong to the tangent space at $x_k$.} :
\begin{equation}\label{expiter}
x_{k+1}=\text{exp}_{x_k}\big(a_{k}\{H(x_{k})+M^{x_k}_{k+1}\}\big).
\end{equation}
We have to be careful while defining the above update because of the \textit{injectivity radius}. We assume throughout the paper that this quantity is bounded below by some $K>0$. Informally, the injectivity radius at $x\in\M$ is the least distance to the cut locus which is where the $\text{exp}_x(\cdot)$ ceases to be the path of minimizing length (see \cite{dcarmo} for details). But since the time step $a_k \to 0$ (see assumption (A2)), this is not a problem as long as the stepsizes are sufficiently small.

The downside of the above scheme is that computing the exponential updates requires solving an ordinary
differential equation which defines the geodesic and is usually
computationally expensive. Even for the simple case of a spherical constraint $S^{n-1}$, the geodesic $t\to x(t)$ expressed as a function of $x_0 \in S^{n-1}$ and $\dot{x}_0 \in \T_{x_0}S^{n-1}$ is given by (Example 5.4.1, \cite{Absil}) :
$$x(t)=x_0\text{cos}(\|\dot{x}_0t\|) +\dot{x}_0\frac{1}{\| \dot{x}_0\|}\text{sin}(\|\dot{x}_0\|t).$$
An alternative is provided by approximating
the geodesics using the concept of retractions defined next.
\begin{defn}
A retraction on $\mathcal{M}$ is a smooth mapping $\mathcal{R}:\mathcal{TM}\to\mathcal{M}$, where $\mathcal{TM}$ is the tangent bundle,
with the following properties :

i) $\mathcal{R}_{x}(0_{x})=x$, where $\mathcal{R}_{x}$ is the restriction of the retraction
to $\mathcal{T}_{x}\mathcal{M}$ and $0_{x}$ denotes the zero element
of $\mathcal{T}_{x}\mathcal{M}$.

ii) {[}Local rigidity{]} With the canonical identification $\mathcal{T}_{0_{x}}\mathcal{T}_{x}\mathcal{M}\simeq\mathcal{T}_{x}\mathcal{M}$,
$\mathcal{R}_{x}$ satisfies
\[
D\mathcal{R}_{x}(0_{x})=\text{id}_{\mathcal{T}_{x}\mathcal{M}},
\]

where $\text{id}_{\mathcal{T}_{x}\mathcal{M}}$ denotes the identity
mapping on $\mathcal{T}_{x}\mathcal{M}$.
\end{defn}
The update (\ref{expiter}) then becomes:
\begin{equation}\label{ret}
x_{k+1}=\mathcal{R}_{x_{k}}\big(a_{k}\{H(x_{k})+M_{k+1}\}\big).
\end{equation}
Again, in the above iteration it is assumed that $M_{k+1}$ belongs to $\mathcal{T}_{x_k} \mathcal{M}$. To compare for the example of the spherical constraint, a possible retraction would be simply to normalise (i.e., divide by the norm) $x_n$ at each step.

\begin{rem}
An important thing to note here is that we do not prescribe the retraction mapping in (\ref{ret}). It is not unique and depending upon the manifold there may be more than one possible choice. In fact the exponential map itself is a retraction (Proposition 5.4.1, \cite{Absil}). Keeping the latter fact in mind, we perform the convergence analysis only for (\ref{ret}). All the results presented also hold for (\ref{expiter}). The specific retraction will usually be chosen based on computational ease.
\end{rem}
We give some examples (Section 3.2, 3.3 \cite{Malick}) of retractions on some matrix manifolds, which will help  highlight the computational appeal of using general retractions as opposed to the exponential map.\\

\noindent{\textbf{Example 1 }(\textit{Projection on Fixed Rank Matrices and Stiefel
Manifolds) }}: Let
\[
\mathcal{R}_{r}=\{X\in\mathbb{R}^{n\times m}:\,\text{rank}(X)=r\}
\]
be the set of matrices with rank $r$ which is known to be a smooth
manifold of $\mathbb{R}^{n\times m}$. Let the singular value decomposition
of any $X \in \mathbb{R}^{n\times m}$ be given by
\[
X=U\Sigma V^{T}
\]
with $U=[u_{1},....,u_{n}]$ and $V=[v_{1},....,v_{m}]$ being orthogonal matrices
and $\Sigma$ being a diagonal matrix having the singular values
of $X$ on its diagonal  in the non-increasing order $\big(\sigma_{1}(X)\geq\sigma_{2}(X)....\geq\sigma_{\min\{n.m\}}(X)\geq0 \big)$.
With a fixed $\tilde{X} \in \mathcal{R}_r$, for any $X$ such that $\|X-\tilde{X}\|<\sigma_{r}(\tilde{X})/2$,
the projection of $X$ onto $\mathcal{R}_{r}$ exists, is unique
and is given by
\[
P_{\mathcal{R}_{r}}(X)=\sum_{i=1}^{r}\sigma_{i}(X)u_{i}v_{i}^{T}.
\]
The Stiefel manifold is defined for $m\leq n$ by
\[
\mathcal{S}_{n.m}=\{X\in\mathbb{R}^{n\times m}\,:\,X^{T}X=I_{m}\}.
\]
Along the same lines as above, let $\tilde{X}\in\mathcal{S}_{n,m}$; then
for any $X$ such that $\|X-\tilde{X}\|<\sigma_{m}(\tilde{X})/2$,
the projection of $X$ onto $\mathcal{S}_{n,m}$ exists, is unique
and is given by
\[
P_{\mathcal{S}_{n,m}}(X)=\sum_{i=1}^{m}u_iv_{i}^{T}.
\]
In particular, this is a retraction.\\

 Since we will be using the ODE method to analyze the algorithm (\ref{ret}),
we first recall the notion of an ODE on a manifold (\cite{Hairer}, Theorem 5.2, Chapter 4):
\begin{defn}
The dynamics
\begin{equation} \label{manODE}
\dot{x}=H(x)
\end{equation}
 defines a differential equation on the manifold when
$$H(x)\in\mathcal{T}_{x}\mathcal{M}\:\text{ for all }x\in\mathcal{M}$$
is a smooth vector field. The existence and uniqueness theorem for ordinary differential equations guarantees that there exists a unique smooth function $\Phi\,:\,\mathbb{R}\times\M \to \M$ such that

i) $\Phi_\cdot (t)\,:\, \M \to \M$ is a diffeomorphism for each $t \in \mathbb{R}$;

ii)  $\Phi_x(t+s) = \Phi_{\Phi_x(t)}(s) $  ; and

iii) for each $x\in \M,$
$$\frac{d}{dt}\bigg\rvert_{t=0}\Phi_x(t) = H(x).$$
\end{defn}
The proof of the existence of a unique smooth function satisfying the above prooperties is given in Chapter 5, \cite{Arnold}. The notion of convergence on manifolds is defined by generalizing
that for the standard Euclidean space using local charts :
\begin{defn}\label{mancon}
An infinite sequence $\{x_{k}\}_{k\geq0}$ of points of a manifold
$\mathcal{M}$ is said to be convergent to a point $x^{*}\in\mathcal{M}$ if there exists a  chart ($\mathcal{U},\psi$)
of $\mathcal{M}$, and a $K>0$
such that $x^* \in \mathcal{U}$, $x_k \in \mathcal{U}$ for $k \geq K$, and $\{\psi(x_{k})\}_{k\geq K}$ converges to the
point $\psi(x^{*})$.
\end{defn}

We make the following important assumption:\\
\begin{enumerate}

\item[(A0)] The injectivity radius at all points in $\M$ is bounded away from zero by some $r_0 > 0$.\\

\item[(A1)] The map  $H\,:\,\mathcal{M} \to \mathcal{T}\mathcal{M}$ is a smooth vector field.\\

\item[(A2)] Step-sizes $\{a(n)\}$ are positive scalars satisfying :
\begin{equation} \label{summability}
\sum_{n}a_n=\infty\,,\sum_{n}a_n^{2} < \infty.
\end{equation}
\item[(A3)] $\{M_{n}\}$ is a martingale difference sequence with respect
to the increasing $\sigma$-fields
\begin{equation}\label{sigma-field-1}
\mathcal{F}_{n}\doteq\sigma(x_0,M_{m},m \leq n),\,n\geq0
\end{equation}
so that $$\mathbb{E}[M_{n+1}|\mathcal{F}_{n}]=0\,\text{ a.s.}$$ Furthermore, we assume that
\begin{equation}
\sup_n\mathbb{E}[\|M_{n+1}\|^{2}|\mathcal{F}_{n}] < \infty. \label{mgbound}
\end{equation}

\item[(A4)] The iterates of (\ref{ret}) remain in a (possibly sample point dependent) compact subset of $\M$, a.s.

\end{enumerate}

\begin{rem}
Note that the norm $\|.\|$ used in (A3) is derived from the Rimeannian metric $\langle \cdot,\cdot\rangle$. Condition (A4) is a stability condition that needs to be separately verified, see \cite{BorkarBook}, Chapter 3 for some tests for stability in the Euclidean case.
\end{rem}

\subsection{Convergence Analysis}

The analysis presented here builds upon the proof for the Euclidean case (Lemma 1, Theorem 2, Chapter 2, \cite{BorkarBook}). Let $t_{n}=\sum_{m=0}^{n-1}a_{m}$ with
$t_{0}=0$. Let $x(t)$ be a continuous time trajectory evolving on the manifold defined
by setting $x(t_{n})=x_{n}$, where $x_n$ is the iterate produced by (\ref{ret}), and then joined by a \textit{geodesic} on the interval
$[t_{n},t_{n+1}).$ Also let $x^{s}(t),\,t\geq s,$ denote the solution of (\ref{manODE}) starting at the point $x(s)$, i.e. $x^s(s)=x(s)$ , so that $x^{t_n}(t), t \geq t_n,$ is the solution to the
ODE (\ref{manODE}) with $x^{t_n}(t_n)=x(t_n)=x_{n}$.

\begin{thm}\label{manthrm}
Suppose (A0)-(A4) hold. Then for any
$T>0$,
\[
\lim_{s\to\infty}\sup_{t\in[s,s+T]} d(x(t),x^{s}(t)) \to0\, \textrm{ a.s. },
\]
where $d(\cdot,\cdot)$ is the Riemannian distance.
\end{thm}
\begin{proof}
We first do the following construction : Cover $\M$ by a countable collection of open geodesic balls of radii $\frac{r_0}{2}$ (i.e., points whose Riemannian distance from a `center' is $< \frac{r_0}{2}$) with the additional properties :
\begin{itemize}

\item [\textbullet]  any point in $\M$ is at a Riemannian distance $< \frac{r_0}{4}$ from the center of at least one such ball,

\item [\textbullet] any compact subset of $\M$ has a finite subcover of such balls.
\end{itemize}
It is easy to see that under (A0), this is always possible. Take a finite subcover from the given collection to cover the compact set in which $x_n$ lie. Let $\{\mathcal{U}_n,\psi_n\}$ denote \textit{a neighborhood} in which the point $x_n$ lies. Under the above construction, the following hold:
\begin{enumerate}
\item Given a pair $(\mathcal{U}, \psi)$ from the chart, without any loss of generality,  both $\psi(\cdot)$ and its differential  $D\psi(\cdot)$ viewed as a map $ D\psi(\cdot): \mathcal{U} \mapsto$ the space of bounded linear operators $\mathcal{T}_{\cdot}\M \to \mathbb{R}^n$ with operator norm, are uniformly bounded on $\UC$. If $\psi(\mathcal{U})$ does not map to a bounded set in $\mathbb{R}^n$ we can take a smaller neighborhood $\mathcal{U'}\subset \mathcal{U}$ with $\overline{\mathcal{U'}} \subset \mathcal{U}$ where $\overline{\mathcal{U'}}$ denotes the closure of $\mathcal{U}'$. The difference $\UC\,/\,\overline{\UC'}$ can again be covered by a finite number of neighborhoods to make sure $\cup_{\alpha}\mathcal{U}_{\alpha} = \M$. The same argument applies to $D\psi(\cdot)$. Note that all of this is possible since we are dealing with geodesic neighborhoods and shrinking such a neighborhoods radius does not violate its diffeomorphic property.  \\

\item For any $x_n$, $d(x_n, y) \geq \frac{r_0}{4} \ \forall \ y \in \partial\mathcal{U}_{n}$, where $\partial\mathcal{U}_{n}$ denotes the boundary of $\mathcal{U}_n$, for at least one such geodesic ball $\mathcal{U}_n$. This property ensures that each $x_n$ lies in the interior of a geodesic ball. Also, this is again without any loss of generality since if this property did not hold true, we would have  an infinite number of $x_n$ with $d(x_n, y) < \frac{r_0}{4} \text{ for some} \ y \in \partial\mathcal{U}_{n}$ and all such $\UC_n$ which contain $x_n$ (for a finite number of such $x_n$ we can just add the geodesic balls centered at them to the chart). But then again, for all such $\UC_n$ which contain these $x_n$ (which are finite in number from compactness), we  can cover the points $y \in \UC_n $ with $d(x_n, y) \geq \frac{r_0}{4}$ with a finite number of geodesic balls of radii $\frac{r_0}{2}$ so that the property is restored. 
\end{enumerate}
The motivation behind this construction will be discussed later. The following notation is used to denote the coordinate expressions :
\[
\hat{x}_{n} =\psi_n(x_{n}),
\]
\[
\hat{H}(\hat{x}_{n}) = D\psi_n(x_{n})[H(x_{n})],
\]
\[
\hat{M}_{n} = D\psi_n(x_{n})[M_n],
\]
where $D\psi_n(x_n) : \mathcal{T}_{x_n}\mathcal{M} \to \mathbb{R}^n$ is the differential of the coordinate mapping $\psi_n$ and is a linear map. Let $\bar{x}(t)$ be a piecewise linear trajectory evolving in $\mathbb{R}^{n}$
defined by setting $\bar{x}(t_{n})=\hat{x}_{n}$ and then doing a
\textit{linear interpolation} on the interval $[t_{n},t_{n+1}).$

Consider the ODE on the manifold expressed in the local
chart $\hat{x}^{t_n}(t)=\psi_n(x^{t_n}(t))$, where $\hat{x}^{t_n}(t)$ denotes the solution to (\ref{newmanODE}) below with $\hat{x}^{t_{n}}(t_{n})=\hat{x}_{n}$, well-defined for $t$ in a sufficiently small interval :
\[
\dot{x}(t)=H(x(t))
\]
\[
\Rightarrow(D\psi(x(t)))\dot{x}(t)=(D\psi(x(t)))H(x(t))
\]
\begin{equation} \label{newmanODE}
\Rightarrow\dot{\hat{x}}(t)=\hat{H}(\hat{x}(t)).
\end{equation}
It is sufficient to show here that $$\lim_{s\to\infty}\sup_{t\in[s,s+T]} \|\bar{x}(t)-\hat{x}^{t_n}(t)\| \to 0,$$ so that we deal with  iterates
on the manifold using local parameterization. This is where Properties (a) and (b) help us. Property (a) helps in making sure that $\hat{x}_n,\,n \geq 0$ remain bounded. We will be comparing trajectories $x(\cdot)$ and $x^{t_n}(\cdot)$ in the interval $[t_n, t_n+T]$  by mapping them to $\bar{x}(\cdot)$ and $\hat{x}^{t_n}(\cdot)$ respectively. So we may have to switch the coordinate charts multiple times while mapping $\{x(t)\,,\,x^{t_n}(t):\,t\in[t_n, t_n+T)$\}. To circumvent this, let us define
$$\bar{T}= \sup_{n\geq0 }\{ t : x(t_n + t ),x^{t_n}( t )  \in \psi_{n}(\mathcal{U}_n) \}. $$
From Property (b)  $d(x_n, y) \geq \frac{r_0}{4} \ \forall \ y \in \partial\mathcal{U}_{n}$, so  $\bar{T}>0$.\footnote{If $\bar{T} = \infty $ (an infinite injectivity radius) or $T < \bar{T}$, all these constructions are  not required.} Note that for $\bar{T}$ thus defined, we do not need to switch coordinate charts while comparing the trajectories $x(\cdot)$ and $x^{t_n}(\cdot)$ in the interval $[t_n, t_n+\bar{T}]$.\\

\noindent{\textit{Claim :}} If the theorem holds for $\bar{T}$, then it does so for any $T\geq\bar{T}$.\\

\noindent{\textit{Proof :}} Let the theorem hold for $T=\bar{T}$, we show that it also holds for $T=2\bar{T}$. We have,
$$\sup_{t\in[s,s+2\bar{T}]} d(x(t),x^{s}(t)) \leq  \sup_{t\in[s,s+\bar{T}]} d(x(t),x^{s}(t)) + \sup_{t\in[s+\bar{T},s+2\bar{T}]} d(x(t),x^{s}(t)).$$
The first term in the above goes to zero in the limit from the assumption. The second term can be handled as follows :
\begin{align*}
\sup_{t\in[s+\bar{T},s+2\bar{T}]} d(x(t),x^{s}(t)) &  \leq \sup_{t\in[s+\bar{T},s+2\bar{T}]} d(x(t),x^{s+\bar{T}}(t)) +\\  & \qquad \qquad \sup_{t\in[s+\bar{T},s+2\bar{T}]} d(x^{s+\bar{T}}(t),x^{s}(t)),
\end{align*}
where $x^{s+\bar{T}}(t)$ satisfies $x^{s+\bar{T}}(s+\bar{T})= x(s+\bar{T})$. The first term in the RHS of the above inequality again goes to zero by assumption (note that the time window is of length $\bar{T}$ and $x^{s+\bar{T}}(t)$ starts at the point $x(s+\bar{T})$). The second term goes to zero because
$$d(x^{s+\bar{T}}(s+\bar{T}),x^{s}(s+\bar{T}))=d(x(s+\bar{T}),x^{s}(s+\bar{T})) \to 0. $$
By the uniqueness and smoothness of the flow, the claim follows. \qed \\

For the rest of the proof we assume that $\bar{T}=T$. Let $t_{n+m}\in[t_{n},t_{n}+T]$, then we have from (\ref{ret}) :
\begin{equation} \label{mantemp}
x_{n+m}=\mathcal{R}_{x_{n+m-1}}\{a_{n+m-1}(H(x_{n+m-1})+M_{n+m})\}.
\end{equation}
The coordinate expression of the above can be written as :
\begin{equation} \label{cordin-expre}
\psi_{n}(x_{n+m}) = \psi_{n}(\mathcal{R}_{x_{n+m-1}}\{a_{n+m-1}(H(x_{n+m-1})+M_{n+m})\}).
\end{equation}
As shown earlier, $x_{n+m}$ belongs to $\mathcal{U}_{{n}}$ without loss of generality, by our choice of $\bar{T}$. A Taylor expansion of the term on the RHS of (\ref{cordin-expre}) gives,
\begin{multline}\label{newmaneq}
 \psi_{n}\big(\mathcal{R}_{x_{n+m-1}}\{a_{n+m-1}(H(x_{n+m-1})+M_{n+m})\} \big)=\\
\begin{aligned}
 \psi_{n}(\mathcal{R}_{x_{n+m-1}}(0_{x_{n+m-1}})) \ + D\psi_{n}(\mathcal{R}_{x_{n+m-1}}(0_{x_{n+m-1}})) \ \times \\
D\mathcal{R}_{x_{n+m-1}} (0_{x_{n+m-1}}) (a_{n+m-1}(H(x_{n+m-1})
+ \ M_{n+m}) \ +\\ \ \text{Rd}(\hat{x}_{n+m-1}).
\end{aligned}
\end{multline}
Here the remainder $\text{Rd}(\hat{x}_{n+m-1})=\mathcal{O}(a_{n+m-1}^{2})$
because the iterates remain in a compact set and $\mathcal{R}(\cdot)$, $H(\cdot)$ are smooth. Thus using
the local rigidity property of a retraction ($DR_{x}(0_{x})=\text{id}_{\mathcal{T}_{x}\mathcal{M}}$) and the fact that $\mathcal{R}_{x}(0_{x})=x$, we can write (\ref{newmaneq}) as
\begin{multline*}
 \psi_{n}(\mathcal{R}_{x_{n+m-1}}\{a_{n+m-1}(H(x_{n+m-1})+M_{n+m})\})= \psi_{n}(x_{n+m-1}) + \\
D\psi_{n}(x_{n+m-1})(a_{n+m-1}(H(x_{n+m-1})+M_{n+m})+ \mathcal{O}(a_{n+m-1}^{2}),
\end{multline*}
which gives
\begin{multline*}
  \psi_{n}(\mathcal{R}_{x_{n+m-1}}\{a_{n+m-1}(H(x_{n+m-1})+M_{n+m})\})= \hat{x}_{n+m-1} \ +\\
  a_{n+m-1}\hat{H}(\hat{x}_{n+m-1})
 +a_{n+m-1}\hat{M}_{n+m}+\mathcal{O}(a_{n+m-1}^{2}),
\end{multline*}
where we used the linearity of the map $D\psi_{n}(\cdot)$ in the last equation. Using this in (\ref{cordin-expre}), we get
\begin{multline*}
\hat{x}_{n+m}=\hat{x}_{n+m-1}+a_{n+m-1}\hat{H}(\hat{x}_{n+m-1})+
a_{n+m-1}\hat{M}_{n+m}+\mathcal{O}(a_{n+m-1}^{2}).
\end{multline*}
Doing a recursion on the above we have
\begin{equation}\label{man-x}
\hat{x}_{n+m}=\hat{x}_{n}+\sum_{k=0}^{m-1}a_{n+k}\hat{H}(\hat{x}_{n+k})+\delta_{n,n+k},
\end{equation}
where
\[
\delta_{n,n+k}=\sum_{k=0}^{m-1}a_{n+k}\big\{\hat{M}_{n+k+1}+\mathcal{O}(a_{n+k})\big\}.
\]
\begin{figure}[H]
\begin{center}
\includegraphics[width=11cm,height=8cm]{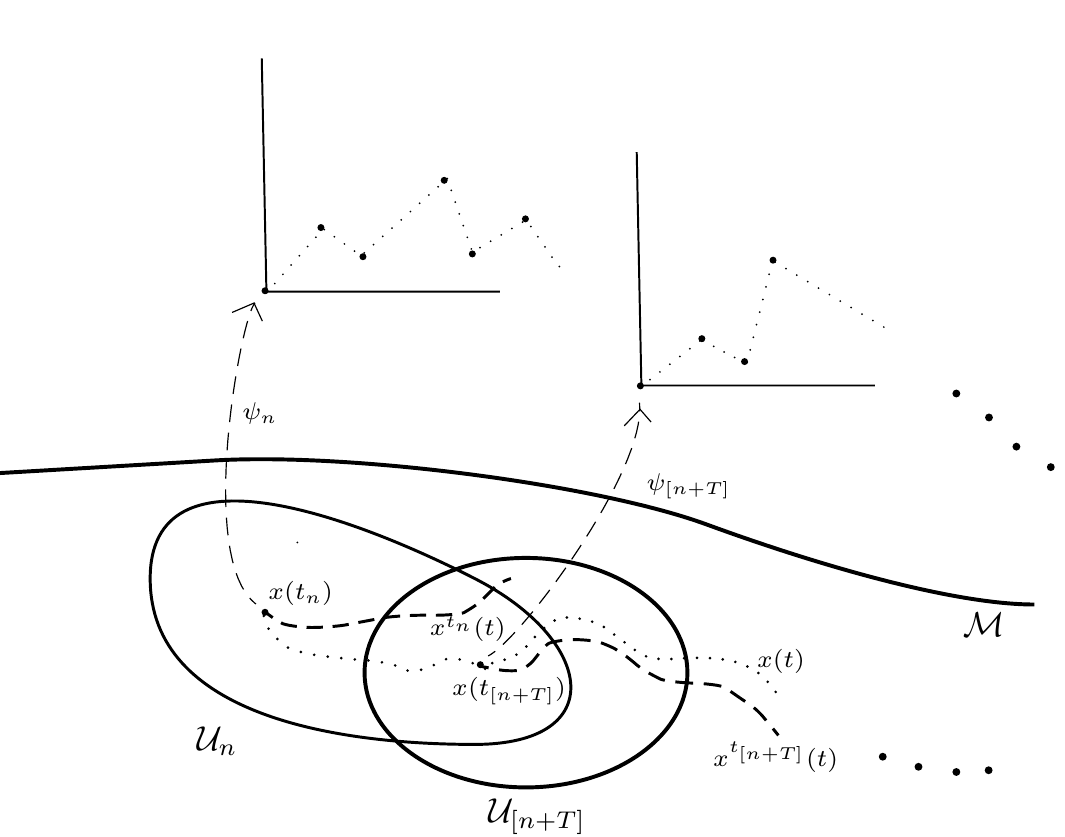}
\end{center}

\caption{This figure depicts an illustration of the local parametrization at $t=t_n$, $t=t_{[n+T]}$} where $[n+T]$ $=\min\{ k : t_k \geq t_n + T \}$. Note that the piecewise linear trajectory corresponds to the iterates $\{\hat{x}_n\}$.

\end{figure}

Note that for $ k \geq 0$, $\hat{M}_{k+1} \in \mathbb{R}^n$ is a martingale since $D\psi(\cdot)$ is a linear operator. Furthermore, we have for any $n \geq 0$,
\begin{align*}
\mathbb{E}[\|D\psi_n(x_n)M_{n+1}\|^{2}|\mathcal{F}_{n}] \leq \|D\psi_n(x_n)\|^2 \mathbb{E}[\|M_{n+1}\|^{2}|\mathcal{F}_{n}],
\end{align*}
where $\|D\psi_n(x_n)\|$ denotes the norm of $D\psi_n(x_n)$. To show that the above quantity is bounded we note that the expectation term in the RHS is bounded by (A3). Also, by (A4) and  the judicious choice of  charts we have made, $\psi_n, D\psi_n$ vary over a (possibly sample path dependent) finite family and therefore are bounded uniformly in $n$ as maps  $\mathcal{U}_{n} \mapsto \mathbb{R}^n$ and $x \in \mathcal{U}_{n} \mapsto$ the space of bounded linear operators $\mathcal{T}_{x}\M \mapsto \mathbb{R}^n$ with operator norm.
So by (A3), (A4), we have $$\sup_n\mathbb{E}[\|D\psi_n(x_n)M_{n+1}\|^{2}|\mathcal{F}_{n}] < \infty.$$
By the martingale convergence theorem (Appendix C, \cite{BorkarBook}), since $\sum_{k\geq0} a_{n+k}^{2}<\infty$, it then follows that $\sum_k a_k \hat{M}_k $ is convergent a.s. This in turn implies that $\sup_{k}\delta_{n,n+k}$ $\to 0$ a.s.\ as $n\to\infty$.

Next we establish similar bounds for the ODE considered in the coordinate expression. Integrating (\ref{newmanODE}) between the limits $t_{n}$ to $t_{n+m}$ we have :
\begin{align}
\hat{x}^{t_{n}}(t_{n+m}) & =\hat{x}^{t_{n}}(t_{n})+\int_{t_{n}}^{t_{n+m}}\hat{H}(\hat{x}^{t_{n}}(t))dt,\nonumber \\
 & =\hat{x}^{t_{n}}(t_{n})+\sum_{k=0}^{m-1}a_{n+k}\hat{H}(\hat{x}^{t_{n}}(t_{n+k})) \nonumber \\
 & +\int_{t_{n}}^{t_{n+m}}\big\{\hat{H}(\hat{x}^{t_{n}}(t))-\hat{H}(\hat{x}^{t_{n}}([t]))\big\} dt, \label{manintbound}
\end{align}
where $[t]=\max\{t_{n}:t_{n}\leq t\}$. We now establish a bound on
the last integral on the right hand side. Since $H(\cdot)$ is smooth we have :
\begin{equation}\label{Hdiff}
\|\hat{H}(\hat{x})-\hat{H}(\hat{y})\|\leq\gamma\|\hat{x}-\hat{y}\|
\end{equation}
for some constant $\gamma$. Using this, we have from integrating (\ref{newmanODE}) between $t_n$ to $t\in [t_n, t_{n}+T)$,
\begin{align*}
\|\hat{x}^{t_{n}}(t)\| & \leq\|\hat{x}^{t_{n}}(t_{n})\|+\int_{t_n}^{t}\|(\hat{H}(\hat{x}^{t_{n}}(t))\|dt,\\
&  \leq \underbrace{\|\hat{x}^{t_{n}}(t_{n})\|+ \| \hat{H}(0)T \|}_{c_T} +\int_{t_n}^{t}\|(\hat{H}(\hat{x}^{t_{n}}(t))-\hat{H}(0)\|dt,  \\
 & \leq c_{T}+\gamma\int_{t_n}^{t}\|\hat{x}^{t_{n}}(t)\|dt, \,\,\, \text{  (from (\ref{Hdiff}))}
\end{align*}
so that by the Gronwall inequality we have
\[
\|\hat{x}^{t_{n}}(t)\|\leq c_{T}e^{\gamma T},\,\,t\in [t_n, t_{n}+T).
\]
Let $C_T \doteq  c_{T}e^{\gamma T} + \| \hat{H}(0)\|$. Integrating (\ref{newmanODE}) between $t_{n+k}$ to $t\in [t_{n+k},t_{n+k+1})$ with $0\leq k \leq m-1$,
\begin{align*}
\|\hat{x}^{t_{n}}(t)-\hat{x}^{t_{n}}(t_{n+k})\| &\leq\|\int_{t_{n+k}}^{t}\hat{H}(\hat{x}^{t_{n}}(s))ds\|,\\
& \leq C_T (t-t_{n+k})\\
& \leq C_T a_{n+k}.
\end{align*}
This gives the required bound :
\begin{align*}
\|\int_{t_{n}}^{t_{n+m}}\big\{\hat{H}(\hat{x}^{t_{n}}(t))-\hat{H}(\hat{x}^{t_{n}}([t]))\big\} dt\| & \leq \int_{t_{n}}^{t_{n+m}}\|\hat{H}(\hat{x}^{t_{n}}(t))-\hat{H}(\hat{x}^{t_{n}}([t]))\|dt,\nonumber \\
 & \leq\gamma\int_{t_{n}}^{t_{n+m}}\|\hat{x}^{t_{n}}(t)-\hat{x}^{t_{n}}([t])\|dt \,\,\, \text{  (from (\ref{Hdiff}))} \nonumber \\
 & \leq\gamma\sum_{k=0}^{m-1}\int_{t_{n+k}}^{t_{n+k+1}}\|\hat{x}^{t_{n}}(t)-\hat{x}^{t_{n}}(t_{n+k})\|dt,\nonumber \\
 & \leq\gamma C_{T}\sum_{k=0}^{m-1}a_{n+k}^{2},\\
\end{align*}
so that
\begin{equation}\label{maninterpol}
\|\int_{t_{n}}^{t_{n+m}}\big\{\hat{H}(\hat{x}^{t_{n}}(t))-\hat{H}(\hat{x}^{t_{n}}([t]))\big\} dt\| \leq\gamma C_{T}\sum_{k=0}^{\infty}a_{n+k}^{2}.
\end{equation}
Subtracting (\ref{man-x}) and (\ref{manintbound}) and taking norms, we get :
\begin{multline*}
\|\hat{x}_{n+m}-\hat{x}^{t_{n}}(t_{n+m})\|\leq\|\hat{x}_{n}-\hat{x}^{t_{n}}(t_{n})\|+\|\delta_{n,n+m}\|\\
\begin{aligned}
&+\|\int_{t_{n}}^{t_{n+m}}\big\{\hat{H}(\hat{x}^{t_{n}}(t))-\hat{H}(\hat{x}^{t_{n}}([t]))\big\} dt\| \\
&+ \sum_{k=0}^{m-1}a_{n+k}\|\hat{H}(\hat{x}_{n+k})-\hat{H}(\hat{x}^{t_{n}}(t_{n+k}))\|.
\end{aligned}
\end{multline*}
Using (\ref{maninterpol}) in the above equation, we get
\begin{multline*}
\|\hat{x}_{n+m}-\hat{x}^{t_{n}}(t_{n+m})\|\leq\|\hat{x}_{n}-\hat{x}^{t_{n}}(t_{n})\|+\sup_{k\geq0}\delta_{n,n+k} +\gamma C_{T}\sum_{k=0}^{\infty}a_{n+k}^{2}\\
\begin{aligned}
& +\sum_{k=0}^{m-1}a_{n+k}\|\hat{H}(\hat{x}_{n+k})-\hat{H}(\hat{x}^{t_{n}}(t_{n+k}))\|.\\
\end{aligned}
\end{multline*}
Set $K_{T,n}\doteq\sup_{k\geq0}\delta_{n,n+k}+\gamma C_{T}\sum_{k=0}^{\infty}a_{n+k}^{2}$. Since $\hat{x}_{n}=\hat{x}^{t_{n}}(t_{n})$ and using (\ref{Hdiff}),
\begin{multline}\label{K-T}
\|\hat{x}_{n+m}-\hat{x}^{t_{n}}(t_{n+m})\|\leq K_{T,n}+ \gamma \sum_{k=0}^{m-1}a_{n+k}\|\hat{x}_{n+k}-\hat{x}^{t_{n}}(t_{n+k})\|.
\end{multline}
Applying the discrete Gronwall inequality (\cite{BorkarBook}, p.\ 146) to the above we get,
\[
\|\hat{x}_{n+m}-\hat{x}^{t_{n}}(t_{n+m})\|\leq K_{T,n}\exp\{\gamma T\}\;\text{a.s.}
\]
Since $K_{T,n}\to0$, the result follows. Note that this proves the theorem only for $t=t_n\uparrow\infty$. However,  we can extend it for general $t\uparrow\infty$  by exploiting the fact that $a_n \to 0$, so that $t_{n+1} -t_n \to 0$.   \qed
\end{proof}

\begin{defn}  (Definition 1.4.1, \cite{Alongi})
A set $A$ is said to be invariant under the flow $\Phi_x(\cdot)$ if $x\in A$ implies that $\Phi_x(t) \in A \, \ \forall t \in \mathbb{R}$.
\end{defn}

\begin{defn} \label{chaintrans}(Definition 2.7.1, \cite{Alongi})
Let $\Phi_x(\cdot)$ be a flow on a metric space $(\M,d(\cdot,\cdot))$. Given $\epsilon>0,\,T>0$, and $x,\,y\in\M$, an $(\epsilon,T)$ chain from $x$ to $y$ with respect to $\Phi_x(t)$ is a pair of finite sequences $x_1=x,x_2,...,x_{n-1},x_n=y$ and $t_1,...,t_{n-1} \in [T,\infty)$ such that
$$ d(\Phi_{x_i}(t_i), x_{i+1} ) < \epsilon$$
for $i=1,..,n-1$. A set $A$ is internally chain transitive for the flow if for any choice of $x, y$ in this set and any $\epsilon, T$ as above, there exists an $(\epsilon, T)$ chain in $A$.
\end{defn}

The following result is a variant of the well known theorem due to Benaim \cite{Benaim0}, a simple exposition of which appears in pp.\ 15-16 of \cite{BorkarBook}, Chapter 2. Taking $\epsilon < \frac{r_0}{2}$ above without loss of generality, an identical argument serves to prove it and is omitted.\\

\begin{thm}\label{manthrm2}
Almost surely, the sequence $\{x_n\}$ generated by (\ref{ret}) converges to a (possibly path dependent) internally chain transitive invariant set of (\ref{manODE}).
\end{thm}

The above theorem establishes the ODE method for manifolds. If $H(\cdot)$ is derived from the Riemannian gradient of a smooth function $\Psi:\M \to \mathbb{R}$ (see Example 2 for definition), then the behavior of the associated ODE, sometimes called the gradient flow, is easy to characterize.  Specifically, the solution converges to a connected component of the set of critical points of $\Psi$ (Appendix C.12, Proposition 12.1, \cite{Helmke}). Since these are the only internally chain transitive sets of the gradient flow, (\ref{ret}) has the same asymptotic behavior.

We conclude this section with some illustrative examples to justify the usefulness of Theorem \ref{manthrm}.\\

\noindent{\textbf{Example 2}\textit{ (Constrained Regression)} :} Suppose we have a stream of data pairs $\{X_n,Y_n\}_{n \geq 1}$ and a measurement model,
$$ Y_n = f_w(X_n) + \epsilon_n, \, n \geq 1, $$
where $X_n \in \mathbb{R}^m $ is the input, $w \in \mathbb{R}^d$ is a parameter which is to be estimated, $Y_n \in \mathbb{R}^k $ is the measurement or output and $\epsilon_n$ is some measurement noise. The family of functions $\{f_w : \mathbb{R}^m \to \mathbb{R}^k : w \in \mathbb{R}^d \}$ parameterized by $w$ is chosen as a `best fit' so as to minimize a suitable scalar penalty on the error $\epsilon$. The outcome depends upon the error criterion selected and one potential option is to use the mean square error which gives the following optimization problem :
$$\min_w \Big\{\Psi(w) \doteq \mathbb{E}[\Phi(w)] \doteq \frac{1}{2}\mathbb{E}[\| Y_n -f_w(X_n) \|^2 ] \Big\}. $$
Suppose in addition to this, we are constrained to stay inside a submanifold $\mathcal{S}$ of $\mathbb{R}^d$. The constraint can be linear, say
\begin{equation}\label{affine}
Lw=c,
\end{equation}
where $L \in \mathbb{R}^{p \times d} $ and $c \in \mathbb{R}^p$, or  it could be nonlinear, for instance, say the $d-1$ dimensional sphere $S^{d-1}$. Note that for this problem the vector field will be the Riemannian gradient of $\Psi\, : \, \mathcal{S} \to \mathbb{R} $. By the Riemannian gradient, $\text{grad }\Psi(x)$ of $\Psi$ at $x$, we mean a vector that satisfies the following two properties :\\

\noindent (a) Tangency Condition : $$\text{grad }\Psi(x) \in \mathcal{T}_x\mathcal{S},  $$

\noindent (b) Compatibility Condition : $$ D\Psi(x) \cdot \xi  = \langle \text{grad }\Psi(x), \xi \rangle \,\text{ for all } \xi \in  \mathcal{T}_x\mathcal{S}.   $$
Since the expectation cannot be evaluated, the only recourse is to replace the exact gradient by the argument of the expetation evaluated at the current guess $w_n$. For the constraint (\ref{affine}) we have (Section 1.6, \cite{Helmke}),
\begin{align*}
\text{grad }\Phi(w)&=(I_d-L^{\dagger}L)\nabla_w \Phi(w)\\
& =(I_d-L^{\dagger}L)\langle Y_n -f_w(X_n), \nabla _w f_w(X_n)  \rangle,
\end{align*}
where $L^\dagger = L^T(LL^T)^{-1}$ denotes the pseudo-inverse with $L^T$ denoting the transpose and $\nabla_w (\cdot) $ is the Euclidean gradient. Then iteration (\ref{ret})
becomes
\begin{equation}
w_{n+1}=\mathcal{R}_{w_{n}}\big(-a_{n}\text{grad }\Phi(w_n) \big). \label{stochgradconstr}
\end{equation}
For the spherical constraint, we have (Section 1.6, \cite{Helmke}),
\begin{align*}
\text{grad }\Phi(w)& =  \nabla_w \Phi(w) - \langle w,\nabla_w \Phi(w)  \rangle w,
\end{align*}
where $\nabla_w \Phi(w) = \langle Y_n -f_w(X_n), \nabla _w f_w(X_n)  \rangle$ as before.

Assuming interchange of expectation and differentiation, a simple calculation shows that the iteration (\ref{stochgradconstr}) is of the form
$$w_{n+1}=\mathcal{R}_{w_{n}}\big(a_{n}(-\text{grad }\mathbb{E}\left[\Phi(w_n)\right] + M_{n+1}) \big),$$
where $\{M_n\}$ is a martingale difference sequence. So by Theorem \ref{manthrm}, we have the limiting ODE  for either case as
$$ \dot{w} = - \mathbb{E}\,[\text{grad }\Phi(w)] = - \text{grad }\Psi(w).$$
Note that the retraction in either case can just be the projection map.\\

\noindent{\textbf{Example 3}\textit{ (Principal Component Analysis)} :}  This problem entails
the computation of the $r$ principal eigenvectors of an $n\times n$ covariance
matrix $A = \mathbb{E}[z_{k}z_{k}^{T}]$ with
$z_{1},...,z_{k},...$ being a stream of uniformly bounded $n$-dimensional data vectors.
Define the cost function :
\[
C(W)=-\frac{1}{2}\mathbb{E}[z^{T}W^{T}Wz]=-\frac{1}{2}\text{Tr}(W^{T}AW)
\]
with $\text{Tr}(\cdot)$ denoting the trace of matrix and $W\in\mathcal{S}_{n,r}$ which becomes a basis of the dominant $r$-dimensional invariant subspace of $A$ when $C(W)$ is minimized. The constraint space of the above minimization problem is the Grassmann Manifold\footnote{This is because $C(W)$ is invariant to rotations $W\to WO$ for any orthogonal matrix $O$.}, $\mathcal{G}(n,r)$ of $r$-dimensional subspaces in a $n$-dimensional
ambient space, i.e., it is the quotient manifold $$ \mathcal{G}(n,r) = \big\{ \mathcal{S}_{n,r} / O(r) \big\}, $$
where $O(r)$ is set of orthogonal $r \times r$ matrices. The noisy Riemannian gradient of $C(W)$ under the sample $z$
is given by $H(z,W)=(I_n-WW^{T})zz^{T}W$, so that iteration (\ref{ret}) becomes :
\begin{equation}\label{Oja}
W_{k+1}=\mathcal{R}_{W_{k}}\{-a_{k}(I_{n}-W_{k}W_{k}^{T})z_{k}z_{k}^{T}W_{k}\}.
\end{equation}
Note that the ODE vector field here is $H(W)=(I_{n}-WW^{T})AW$. From Theorem \ref{manthrm} the relevant ODE is
\[
\dot{W}=(I_{n}-WW^{T})AW.
\]
This is the celebrated Oja's algorithm \cite{Oja}. To see how the above can be interpreted as an ODE on $\mathcal{G}_{n,r}$, see (\cite{Hairer}, Example 9.1). The ODE equilibrium points correspond to
\begin{align*}
AW&=WW^{T}AW \\
\Rightarrow AW &= WM,
\end{align*}
for $M=W^TAW$, proving that any limit $W$ is an invariant subspace of $A$. A retraction on $\mathcal{S}_{r,n}$ which could be used in (\ref{Oja}) is given by (Example 4.1.3, \cite{Absil}) :
$$ \mathcal{R}_{W}(a H) = \text{qf}(W + aH),$$
where $\text{qf}(\cdot)$ gives the orthogonal  factor in the QR-decomposition of its argument.

\section{Approximate Retractions on Submanifolds}

\subsection{ Sub-manifolds as constraint sets}
In this subsection we consider the case of the constraint being an embedded submanifold of a Euclidean space. The setup is as follows : The constraint space is a submanifold $\mathcal{S}$ of
class $\mathcal{C}^{k}$ ($k\geq2$) of a Euclidean space $\mathcal{E}$ of dimension $n$. A submanifold of class $\mathcal{C}^{k}$ and  dimension $d$ here means that $\mathcal{S}$ is locally
a coordinate slice, that is, for all $x\in\mathcal{S}$, there exits
a neighborhood $\mathcal{U}_{\mathcal{E}}$ of $x$ in $\mathcal{E}$
and a $\mathcal{C}^{k}$ diffeomorphism $\psi$ on $\mathcal{U}_{\mathcal{E}}$
into $\mathbb{R}^{n}$ such that
\[
\mathcal{S}\cap \mathcal{U}_{\mathcal{E}}=\{x\in\mathcal{U}_{\mathcal{E}}\,:\,\psi_{d+1}(x)=\cdot\cdot\cdot=\psi_{n}(x)=0\}.
\]
 As before, assume (A0)-(A4). The only modification now is that the vector field is now defined on the entire embedding space $\E$, i.e. $H\,:\,\E \to \E$ and instead of being smooth, we only require it to be Lipschitz. Also, instead of  using coordinate charts to prove the relevant results, we can directly use the norm $\|\cdot\|$ inherited from the embedding space $\E$. In addition, we assume the following.

(A5) There exists an $f:\mathbb{R}^{n}\to\mathbb{R}^{n}$  such that for $f^n \doteq f\circ f\circ  \cdots \circ f$ ($n$ times), $P(x)\doteq\lim_{n\to\infty}f^{n}(x$)
exists for all $x\in\mathbb{R}^{n}$. We assume the following conditions on $f$ and $P$ :
\begin{enumerate}

\item[(I)] $f$ is continuous,

\item[(II)] $f^{n}\to P$ uniformly on compacts, so that $P$ is continuous.

\item[(III)] $P$ is the Euclidean projection onto the submanifold $\mathcal{S}$, i.e.
\[
\mathbf{}P(y)=\textrm{arg}\min_{x\in \SA }\|y-x\|.
\]
\end{enumerate}
The aim of this subsection is to provide a framework to do projected SA on submanifolds. The idea is to use $f$ to effect a projection to a differentiable submanifold. This will enable us to execute stochastic approximation versions of algorithms on constraint sets such as matrix manifolds. We will provide some examples later in the section on how to choose such a $f$ for certain sub-manifolds. In case such a $f$ does not exist, one can replace $f$ with the exact retraction operation (see Lemma \ref{manproj} below), but the main point of the presented framework is to show that one need not use exact projections while doing projected SA.

We first note an interesting connection between retractions and projections for submanifolds, which  we recall in a lemma borrowed from \cite{Malick}.  For a submanifold $\mathcal{S}$, the projection
is interpreted in the usual way with the metric taken to be the standard
Euclidean metric of the ambient space $\mathcal{E}$. We state
the following result from (\cite{Malick}, Lemma 4) :
\begin{lem}\label{manproj}
Let $\mathcal{S}$ be a submanifold of class $\mathcal{C}^{k}$ around $x\in\mathcal{S}$
and $P_{\mathcal{S}}$ denote the projection onto $\mathcal{S}$.
Then $P_{\mathcal{S}}$ is well defined (always exists and is unique)
locally and the function $P_{\mathcal{S}}$ is of class $\mathcal{C}^{k-1}$
around $x$ with
\[
DP_{\mathcal{S}}(x)=P_{\mathcal{T}_x\mathcal{S}},
\]
where $P_{\mathcal{T}_x\mathcal{S}}$ is the orthogonal projection onto $\mathcal{T}_x\mathcal{S}$, the tangent space at point $x$. Also, the
function $\mathcal{R}:\mathcal{TS}\to\mathcal{S}$ defined by $(x,u)\longmapsto P_{\mathcal{S}}(x+u)$
is a retraction around x. 
\end{lem}
We propose the following iteration to do (projected) stochastic approximation on embedded submanifolds:
\begin{equation}\label{nongos}
x_{k+1} = f(x_k) + a_k(H(x_k) + M_{n+1}).
\end{equation}
Note that in the above iteration, feasibility is satisfied only asymptotically as proved in Lemma 10, \cite{Mathkar}. The analysis  of this algorithm is easy for submanifolds (of class $\mathcal{C}^k,\,k\geq2$) as we can simply take the Taylor expansion of (\ref{nongos}) after projecting it   :
$$  P(x_{k+1}) = P(f(x_k) + a_k(H(x_k) + M_{n+1}) ).$$
Setting $\tilde{x}_k = P(x_k)$ we have,
$$  \tilde{x}_{k+1} = P(f(x_k)) + a_k DP_{f(x_k)}(H(x_k) + M_{n+1}) ) +\mathcal{O}(a^2_k).$$
The rest of the analysis is similar to \cite{Mathkar}, and noting from Lemma \ref{manproj}  that $DP_{\mathcal{S}}(x)=P_{\mathcal{T}_x\mathcal{S}}$ and $f(x)=x$ for any limit point, the limiting behaviour of (\ref{nongos}) is similar to the following projected ODE :
\begin{equation}\label{Fret}
\dot{x} = P_{\mathcal{T}_x\mathcal{S}}(H(x(t))).
\end{equation}
We give some example's that fit this framework :\\

\noindent{\textbf{Example 4}\textit{ (Affine Manifold)} :} We first examine the simple case of an affine manifold, i.e.,
$$\SA = \{x \in \mathbb{R}^n \,:\,Ax=b\},$$
where $A \in \mathbb{R}^{m\times n}$ and $b\in \mathbb{R}^m$. Note that this manifold can be obtained as
$$\SA = \bigcap_{i=1}^m \SA_i, $$
where for $a_i \doteq$ the $i$ th row of $A$,
$$ \SA_i \doteq \{ x \in \mathbb{R}^n \,:\, a_i^T x=b_i \}.    $$
One way to project onto this set is to use alternating projections.  Then if $P^i$ denotes the projection onto $\SA_i$,  a single full iteration of  alternating projections can be used as a substitute for $f$, so that
 $$ f(\cdot) = P^m \circ \cdots \circ P^1(\cdot).$$\\

\noindent{\textbf{Example 5}\textit{ (Manifold of Correlation Matrices)} :} Suppose we have a constraint minimization problem of the following kind :
\begin{align*}
\min_{C \in \mathbb{R}^{n\times n}} & \quad f(C) \\
\text{s.t.} \,\,\,\,\,\,& C=C^T,\, C \succeq 0,\,C_{ii}=1\,(i=1, \cdots, N),
\end{align*}
where $C \succeq 0$ for a symmetric $C$ denotes that $C$ is positive semidefinite. The constraint set is the set of symmetric correlation matrices, i.e., the set of symmetric positive semidefinite matrices with diagonal elements equal to one.
Note that the constraint can be obtained as the intersection of the following two convex sets :
$$ S \doteq \{ C=C^T \in \mathbb{R}^{n \times n} : C \succeq 0 \},  $$
$$U \doteq \{ C=C^T \in \mathbb{R}^{n \times n} : C_{ii} = 1, \, i=1, \cdots, n \}.$$
The projection on both of these sets is well defined (Theorem 3.1, Theorem 3.2, \cite{Higham}). One can use the Boyle-Dykstra-Han algorithm (see Algorithm 3.3, \cite{Higham}) to project onto the above constraint set. As in the previous example, one can use a single full iteration of this algorithm as a substitute for $f$ in (\ref{non}).\\

\noindent{\textbf{Example 6} : Suppose we have  a constraint set of the following form which arises quite often in control theory (see \cite{Lew}) :
$$ \R_r \cap \M_\A,$$
where $\R_r$ is as in Example 1 and
$$ \M_A \doteq \{ X \in \mathbb{R}^{n \times m} \,:\, \A(X) =b\} $$
for a given linear map $\A :\mathbb{R}^{n\times m} \to \mathbb{R}^d$ and vector $b \in \mathbb{R}^d $. Note that the projection on both sets is well defined and the intersection is a manifold since $\SA_r$ and $\M_A$ intersect transversally, i.e.,
$$\T_{\R_r} + \T_{\M_\A}  = \mathbb{R}^{n \times m}.$$
Again, one could adapt the subroutine suggested in the previous examples.

\begin{rem}
One of the pitfalls of this scheme is that the ODE approximation to (\ref{Fret}) requires Frechet differentiability of the projection map (see assumption (B3), \cite{Mathkar}) which, although satisfied for submanifolds of class $\mathcal{C}^k,\,k\geq2$, is not satisfied for constraint sets with nondifferentiable boundaries. This motivates  an alternative which we explore in the next subsection.
\end{rem}

\subsection{General Constraint Sets}
In continuation of the framework presented in the previous subsection we extend it to general constraint sets which may not be differentiable manifolds. The justification for this is that the projection map is frequently not differentiable for constraints in optimization or stochastic approximation, so the results of the previous subsection will not hold. Here we consider more general sets such as polyhedra, intersection of convex sets which have boundaries that may be nonsmooth or nonconvex sets.

We propose the following iteration to do (projected) stochastic approximation :
\begin{equation}\label{non}
x_{k+1} = (1 - \gamma_k)x_k + \gamma_k f(x_k) + a_k (H(x_k) + M_{k+1}),
\end{equation}
where $ a_k = o(\gamma_k)$ and both satisfy Assumption (A2). The rest of the assumptions are the same as in subsection 3.1 with one addition, viz., we require $f$ to be non-expansive, i.e. for any $x,y\in\E$
$$\|f(x)-f(y)\|\leq\|x-y\|.$$

We first provide some examples that fit the above framework.\\

\noindent{\textbf{Example 7} \textit{ (Polyhedral Set)} :}  The most important example of a constraint set that fits this setup is the polyhedral set given by :
$$\mathcal{X} = \bigcap_{i=1}^m \mathcal{X}_i, $$
where
$$ \mathcal{X}_i \doteq \{ x \in \mathbb{R}^n \,:\, a_i^T x \leq b_i \}.    $$  For this constraint set, the Boyle-Dykstra-Han algorithm  produces an iterate which lies in the convex hull of the previous iterate and its projection (Theorem 2.4, \cite{Deut2}), i.e.,
\begin{equation}
f^{i}(x)=(1-\lambda)x +\lambda P^{i} (x),
\end{equation}
where $\lambda \in (0,1)$. We can take $f(\cdot)=f^1\circ \cdots\circ f^m (\cdot)$ as in Example 4. Note that each $f^i$ is non-expansive here.\\

\noindent{\textbf{Example 8} \textit{ (Two time scale projection)} :}  Another approach to tackle the projection which falls under this setup was proposed in \cite{SuhBor}. We mention it briefly here. It involves running the following iteration :
$$x_{n+1} = (1-a_k)x_k + a_k z_k + a_k (H(x_k) +M_{k+1}), $$
where $z_k$ is obtained from a sub-routine running on a faster time scale and satisfies $\|z_k  -P(x_k)\|\to 0 $. For more details on this algorithm see \cite{SuhBor}.\\

Let  $d(x,A)=\inf_{y \in \SA}\|x-y\|$ be the distance of a point $x$ from the constraint set $\SA$ which we assume to be closed. Since we are no longer dealing with smooth constraints, we have to deal with more general notions of tangent and normal spaces. We define the tangent cone and the normal cone for our constraint set $\SA$ as :

\begin{defn} 
For $x \in \SA$, we define by
$$T_{\mathcal{S}}(x) \doteq  \{v \in \mathbb{R}^n \,|\, \lim_{h\to 0^+}\inf \text{d}(x + hv, \SA)/h = 0\},$$
the tangent cone to $\SA$ at $x$ and by
\begin{equation}\label{cone}
N_{\mathcal{S}}(x) \doteq  \{p \in \mathbb{R}^n \,|\, \langle p,v \rangle \leq 0 \, \forall \, v \in   T_{\mathcal{S}}(x) \},
\end{equation}
the normal cone to $\SA$ at $x$.
\end{defn}
$T_{\mathcal{S}}(x)$ is sometimes also called the contingent cone. Also, for later use, we define the following set, 
\begin{equation*}
N_{\text{prox}(\mathcal{S})}(x) \doteq  \{v \in \mathbb{R}^n \,|\, v \in P^{-1}(x) -x \}.
\end{equation*}
This set is referred to as the proximal normal cone. It contains vectors $v \in N_{\text{prox}(\mathcal{S})}(x) $, such that $x \in P(x+ \tau v)$ for some $\tau >0$. We have (Lemma 6.1.1 \cite{Aubin1}),
\begin{equation}\label{proxcone}
  N_{\text{prox}(\mathcal{S})}(x) \subset N_{\mathcal{S}}(x).
\end{equation}
The inclusion can be strict, even for sets defined as smooth inequalities (Section E, Chapter 6, \cite{Rock}). The relevant differential inclusion for our problem is
\begin{equation}
\dot{x}\in H(x(t))-N_{\mathcal{S}}(x(t)), \label{inclusion2}
\end{equation}
\[
x(t)\in\mathcal{S}\, \ \ \forall t\in[0,T],
\]
where $H(\cdot)$ is the vector field under consideration. For convex sets, this inclusion is identical to the well known `` Projected Dynamical System" considered in \cite{Nagurney} :
\begin{equation*}
\dot{x}=\Pi_{T_{\mathcal{S}}(x)}(x,H(x)),
\end{equation*}
where $\Pi_{T_{\mathcal{S}}(x)}(x,H(x))$ is defined to be the following limit for any $x\in\mathcal{S}$ :
\[
\Pi_{T_{\mathcal{S}}(x)}(x,H(x))\doteq\lim_{\delta\to0}\frac{P(x+\delta H(x))-x}{\delta}.
\]
This map gives the projection of $H(x)$ on the tangent cone $T_{\mathcal{S}}(x)$. For arbitrary closed sets, $T_{\mathcal{S}}(x)$ is replaced by its convex hull, $\text{conv}\{T_{\mathcal{S}}(x)\}$. The proof of the fact that the operator $\Pi_{T_{\mathcal{S}}(x)}(\cdot,\cdot)$ is identical to the RHS of (\ref{inclusion2}) is provided in (\cite{Dupuis}, Lemma 4.6) for convex sets and in (\cite{Serea}, Proposition 5) for arbitrary closed sets.

We recall the following notion  from \cite{Benaim} where more general differential
inclusions are considered:

\begin{defn}(\cite{Benaim})
Suppose $F$ is a closed set valued map such that $F(x)$ is a non-empty compact convex
set for each $x$. Then a perturbed solution
$y$ to the differential inclusion
\begin{equation}
\dot{x}\in F(x) \label{DifInc}
\end{equation}
is an absolutely continuous function which satisfies:\\

\noindent i) $\exists$ a locally integrable function $t\to U(t)$ such that
for any $T>0$, $$\lim_{t\to\infty}\sup_{0\leq v\leq T}\big|\int_{t}^{t+v}U(s)ds\big|=0,$$
ii) $\exists$ a function $\delta:[0,\infty)\to[0,\infty)$ with $\delta(t)\to0$ as $t \to \infty$
such that
\[
\dot{y}-U(t)\in F^{\delta(t)}(y),
\]
where $F^{\delta}( \cdot )$ denotes the the following set :
\[
F^{\delta}(x)=\big\{ z\in\mathbb{R}^{n}\,:\,\exists x' \,\textrm{such that}\, \ \|x-x'\|<\delta.  \ \text{d}(z,F(x'))<\delta\big\}
\]
\end{defn}

There is no guarantee that the perturbed solution remains close to a
solution of (\ref{DifInc}),  however, the following assumption  helps in establishing some form of convergence. Let $\Lambda$ denote the equilibrium set of (\ref{inclusion2}), assumed to be non-empty. Then :
\textit{
\[
\Lambda \subset \{x\,:\,H(x)\in N_{\mathcal{S}}(x)\}.
\]
}
Assume:

\noindent (A6) There exists a Lyapunov function for
the set $\Lambda$, i.e.,  a continuously differentiable function $V:\mathbb{R}^{n}\to\mathbb{R}$
such that any solution $x$ to (\ref{inclusion2}) satisfies
\[
V(x(t))\leq V(x(0)) \ \forall \\ t > 0
\]
 and the inequality is strict whenever $x(0)\notin \Lambda$.}\\

(A6) is satisfied, for instance if $H=-\nabla g$ for some continuously differentiable $g$. Then $\Lambda$ represents the KKT points and the function $g$ itself will serve as a Lyapunov function for the set $\Lambda$.  If (A6) does not hold, one can still say the following. The asymptotic behavior is analogous to Theorem \ref{manthrm2}. Specifically, we already know that under reasonable assumptions, an SA scheme converges a.s.\ to an internally chain transitive invariant set of the limiting o.d.e. This behavior extends to differential inclusions as well (see \cite{Benaim} or Chapter 5, \cite{BorkarBook}) and is what we would expect for  this algorithm if we remove (A6). It is important, however, that an `invariant' set should be defined in the weak sense, i.e., any point in it should have at least one trajectory (as opposed to \textit{all} trajectories) of the differential inclusion passing through it, that remains in the set for all time. 

The following result is from (\cite{Benaim}, Prop. 3.27) :
\begin{prop} \label{Perturbed}
Let $y$ be a bounded perturbed solution to (\ref{inclusion2}) and there exist a Lyapunov
function for a set $\Lambda$ with $V(\Lambda)$ having an empty interior. Then
\[
\bigcap_{t\geq0}\overline{y\big([t,\infty)\big)}\subset\Lambda.
\]
\end{prop}

\noindent (A7) The following condition is satisfied
$$\lim_n\sum_{j=n}^{n+m(n)} \gamma_j \|f(x_j)-x_j \| \to 0$$
with, $$m(n) := \min\{k \geq n : \sum_{j=n}^{n+k}a_j \geq T\}, n \geq 0.$$

Before, proceeding further, we discuss assumption (A7) since it seems a bit restrictive. To get some intuition regarding this assumption we first note that (\ref{non}) draws a parallel with the Tsitsiklis scheme \cite{Tsitsiklis}. The fixed points ($f(x)=x$)  in the Tsitsiklis scheme  belong to the consensus subspace which is the set of fixed points of some doubly stochastic matrix (say $Q$), i.e. the set $Qx=x$ (see \cite{Nedic}). Convergence to this set takes place at a linear rate which is a direct consequence of the fact that the product of doubly stochastic matrices converges to its stationary average at a linear rate. This helps in ensuring that the error terms $\|Qx_k-x_k\|$ stay sufficiently small (see Lemma 1, \cite{Mathkar}).  However, we do not have such a condition here, in particular the convergence rate of $f^n (\cdot) \to P (\cdot)$ is not assumed to be linear. But all we need is (A7) and it can be ensured by making sure that the term $\|f(x_k)-x_k\|$ becomes sufficiently small. For instance, for Example 7 this can be done by replacing $f$ with a time dependent $f_n$,
$$f_{n}(\cdot) = \underbrace{f\circ..\circ f}_{i_n \text{  times}}(\cdot), $$
where $i_n$ denotes a sufficient number of Boyle-Dykstra iterations, so that the term $\|f(x_k)-x_k\|$ stays small \cite{SuhBor}.

We have to also make sure that there is at least one solution to (\ref{inclusion2}).  A class of sets for which this holds is (bounded) proximal retracts which include convex sets, sets with continuously differentiable boundaries etc. A proximal retract set is defined as follows :
\begin{defn} \cite{Serea}
 A closed set $\mathcal{S}\subset\mathbb{R}^{n}$
is called a proximal retract if there is a neighborhood $N$ of $\mathcal{S}$
such that the projection $P(\cdot)$ is single valued
in $N$.
\end{defn}

\noindent (A8) The set $\SA$ is a bounded proximal retract.\\

\begin{rem}\label{remark1}
To prove the main result, we shall show that the suitably interpolated iterates generated
by (\ref{non}) form a perturbed solution to the differential inclusion
(\ref{inclusion2}), so that using Proposition \ref{Perturbed} (along with assumption
(A6)), the algorithm is shown to converge to its
equilibrium set.\\

\end{rem}

\begin{thm}\label{pert.}
If (A1)-(A8) hold,  we have
\[
x_n\to\{x\,:\,x\in\Lambda\} \ \mbox{a.s.},
\]
where $\Lambda$ denotes the set :
\[
\Lambda =\{x\in\mathcal{S}\,:\,H(x)\in\mathcal{N}_{\SA}(x)\}.
\]
\end{thm}
\begin{proof}
The proof presented here uses the
techniques of \cite{Benaim}, \cite{Bianchi}. Write (\ref{non}) as :
\begin{equation}\label{thrmeq}
x_{k+1} = (1 - \gamma_k)x_k + \gamma_k (f(x_k)  + \epsilon_k + \hat{M}_{k+1}),
\end{equation}
where,
$$\epsilon_k = \frac{a_k}{\gamma_k}(H(x_k)) \text{  and   }  \hat{M}_{k+1}=  \frac{a_k}{\gamma_k}M_{k+1}.$$
Note that $\epsilon_k \to 0$ since $H(x)$ is bounded from (A1),(A4) and $a_k=o(\gamma_k)$. Also, $\hat{M}_{k+1}$ is just a scaled martingale. We first prove that $x_n \to \SA$.\\

\noindent \textit{Claim :}
$\lim_k \inf_{z\in \SA}\|x_{k}-z\| = 0\, $\\

\noindent \textit{Proof.}  Let us first consider the following fixed point iteration :
\begin{equation}\label{xtildeeq}
\tilde{x}_{k+1}=(1-\gamma_{k})\tilde{x}_{k}+ \gamma_{k}f(\tilde{x}_{k}).
\end{equation}
By the arguments of \cite{BorkarBook}, Chapter 2, this has the same asymptotic behavior as the o.d.e.
\[
\dot{\bar{x}}=f(\bar{x})-\bar{x}.
\]
Consider the Lyapunov function $V(\bar{x})=\frac{1}{2}\|\bar{x}\|^{2}$. W.l.o.g. assume that $0$ is a fixed point of $f(\cdot)$.\footnote{If this does not hold, we use the Lyapunov candidate $V(\bar{x})=\frac{1}{2}\|\bar{x}-x^*\|^{2}$ for some fixed point $x^*$ of $f$.} 
Then
\begin{align*}
\frac{d}{dt}V(\bar{x}(t)) & =\bar{x}(t){}^{T}(f(\bar{x}(t))-\bar{x}(t))\\
 & =\bar{x}(t){}^{T}f(\bar{x}(t))-\|\bar{x}(t)\|^{2}.
\end{align*}
For any $v\in\mathbb{R}^{n}$, using the non-expansive property of $f$ and the fact that $0$ is a fixed of $f$ leads to
\[
\|f(v)\|\leq\|v\|.
\]
By the Cauchy-Schwartz inequality,
\[
\frac{d}{dt}V(\bar{x}(t))\leq 0.
\]
By Lasalle's invariance principle we have any trajectory $\bar{x}(\cdot)$  converging to the largest invariant set where $\frac{d}{dt}V(\bar{x}(t))=0$,
which is precisely the set of fixed points of $f$. The claim for $\{\tilde{x}_k\}$ now follows by a standard argument as in Lemma 1 and Theorem 2, pp.\ 12-16, \cite{BorkarBook}. To prove the claim, we simply note that (\ref{thrmeq}) and (\ref{xtildeeq}) only differ by an $o(1)$  perturbation and the noise term $\hat{M}_k$, which does not affect the convergence to the fixed ponts of $f$ (see remark, p.\ 17, \cite{BorkarBook}). This implies in particular that any limit point of (\ref{xtildeeq}), say $x^*$, satisfies
$$ f(x^*)=x^*,$$
$$\Rightarrow  P(x^*)=\lim_nf^n(x^*)=x^* .$$

Let $t_{0}=0$ and $t_{k}=\sum_{i=0}^{k-1}a_{i}$ for any $k\geq1$,
so that $t_{k+1}-t_{k}=a_{k}$. Define the interpolated trajectory
$\Theta:[0,\infty)\to\mathbb{R}^{n}$ as :
\[
\Theta(t)=x_{k}+(t-t_{k})\,\frac{x_{k+1}-x_{k}}{t_{k+1}-t_{k}},\ t\in[t_{k},t_{k+1}],\ k\geq1
\]
 By differentiating the above we have
\begin{align}\label{der.}
\frac{d\Theta(t)}{dt}= & \frac{x_{k+1}-x_{k}}{a_{k}}\,\,\,\,\forall\,t\in[t_{k},t_{k+1}],
\end{align}
where we use the right, resp., left derivative at the end points. Thus by setting $Z_{k} \doteq \frac{\gamma_k}{a_k} (f(x_{k})-x_{k})$ and $W_{k}\doteq x_{k}-P(x_{k})$, we get the following from (\ref{non}) and (\ref{der.}) :
\begin{equation}
\frac{d\Theta(t)}{dt} \in H(x_{k})-W_{k}+W_k+Z_{k}+M_{k+1}. \label{theta1}
\end{equation}
We note here that $W_k \in N_{\text{prox}(\mathcal{S})}(x) (P(x_k)) \subset N_\mathcal{S}(P(x_k))$. Let $F(\cdot)$ denote the following set valued map
\begin{equation}
F(x)=\{H(x)-W\,:\, W \in N_\mathcal{S}(x)\,, \ \|W\|\leq K\},\label{eff}
\end{equation}
 where $0<K<\infty$ is a suitable constant. Note that,
$$
F(x_k)=H(x_{k})-W_{k} \subset F^\delta(P(x_k)) \subset F^{\delta +\delta' } (\Theta(t)),
$$
where $\delta = \|x_k -P(x_k) \|$ and $\delta' = \|P(x_k)-\Theta(t) \| $. We have used the following fact in the above - for any set valued
map $\tilde{H}(\cdot)$ we have
\[
\forall\,(x,\hat{x})\ \in\mathbb{R}^{n}\times\mathbb{R}^{n},\;\!\ \tilde{H}(x)\subset \tilde{H}^{\|x-\hat{x}\|}(\hat{x}).
\]
To finish the proof, $\Theta(\cdot)$ is first shown to be a perturbed solution of a projected dynamical system. Let $\eta(t)\doteq\|x_k -P(x_k) \|+\|\Theta(t)-P(x_{k})\|,\,t\in[t_{k},t_{k+1}),k\geq0$.
Also define
\[
U(t)=U_{k}:=Z_{k}+W_{k+1}+M_{k+1},
\]
 for $t\in[t_{k},t_{k+1}),k\geq0$.
 Then we have from (\ref{theta1}),
\begin{equation}
\frac{d\Theta(t)}{dt}-U(t)\in F^{\eta(t)}(\Theta(t)).\label{incl1}
\end{equation}
If we show that $\eta(t)\to0$ and $\sum_{k}a_{k}U_{k}<\infty$,
we are done by Proposition \ref{Perturbed} and Remark
\ref{remark1}. Then, by (\ref{incl1}), $\Theta(\cdot)$ can be interpreted
as a perturbed solution of the differential inclusion
\[
\dot{\Psi}(t)\in F(\Psi(t)).
\]
 Convergence to the set $\Lambda$ then follows as discussed in Remark \ref{remark1}. We first prove that :
\[
\lim_{t\to\infty}\sup_{0\leq v\leq T}\big|\int_{t}^{t+v}U(s)ds\big|=0.
\]
 We have
\[
\int_{t_k}^{t}U(t)=a_{k}\big[\underbrace{M_{k+1}}_{I}+\underbrace{Z_{k}+W_k}_{II}\big]
\]
 for $t=t_{k+1}$. For $T > 0$, let $$m(n) := \min\{k \geq n : \sum_{j=n}^{n+k}a_j \geq T\}, \ n \geq 0.$$\\
I: This term is the error induced by the noise. The process $\sum_{m=0}^{k-1}a_m\times$ $M_{m+1}, k \geq 1,$ is a zero mean square integrable martingale w.r.t.\ the increasing $\sigma$-fields defined in (A3), with   $\sum_{m=0}^{\infty}a_m^2\mathbb{E}\left[\|M_{m+1}\|^2 | \mathcal{F}_m\right] < \infty$ from (A3) and the square-summability of $\{a_m\}$.  It follows from  the martingale
convergence theorem (Appendix C, \cite{BorkarBook}), that this martingale converges a.s. Therefore
\[
\sup_{\ell \leq m(n)}\|\sum_{k=n}^{\ell}a_{k}M_{k+1}\|\to0\,\,\,\text{a.s. }
\]
II: The fact that $a_kZ_k \to 0$ is simply assumption (A7). Note that
\begin{align*}
\sup_{\ell \leq m(n)}\|\sum_{k=n}^{\ell} a_kZ_k\| &= \sup_{\ell \leq m(n)}\|\sum_{k=n}^{\ell} \gamma_k (f(x_k) -x_k)\|,\\
& \to 0.
 \end{align*}
From the claim proved previously we have $\|x_k -P(x_k)\|\to0$, so that 
\begin{align*}
\sup_{\ell \leq m(n)}\|\sum_{k=n}^{\ell} a_k W_k\| &= \sup_{\ell \leq m(n)}\|\sum_{k=n}^{\ell} a_k (x_k-P(x_k) )\|,\\
& \leq  \sup_{\ell \leq m(n)} \|x_{\ell} -P(x_{\ell}) \| \cdot \Big( \sup_{\ell \leq m(n)}\sum_{k=n}^{\ell} a_k \Big) ,\\
& \leq  \sup_{\ell \leq m(n)} \|x_{\ell} -P(x_{\ell}) \|\cdot T\\
& \to 0.
 \end{align*}

All that is left to prove that the theorem holds is that $\eta(t) \to 0$. By the definition of $\Theta(\cdot)$,
for any $t\in[t_{k},t_{k+1}]$ :
\begin{align*}
\|\Theta(t)-P(x_{k})\| & \leq \|x_{k}-P(x_{k})\| \,+\|x_{k+1}-x_k\|\frac{\|t-t_{k}\|}{\|t_{k+1}-t_{k}\|},\\
 & = \|x_{k}-P(x_{k})\|+  \mathcal{O}(\gamma_{k}), \\
 & \to 0,
\end{align*}
using the claim proved above. Thus $\eta(t) =  \|x_{k}-P(x_{k})\|+ \|\Theta(t)-P(x_{k})\| \to 0 $.   This completes the proof.     \qed
\end{proof}

\section{Conclusions}

To conclude we note some important connections with related works in order to suggest some possible extensions and future research directions.

\begin{enumerate}[label=(\roman*)]

\item \textit{Newton like schemes for stochastic optimization } : Second order methods have attracted a lot of attention recently for stochastic optimization, particularly in the field of machine learning. The same trend has been observed in the field of Riemannian optimization, most notably in \cite{Qi}, \cite{Ring} and \cite{Wen}  where quasi-Newton methods are studied. The same level of attention as the first order methods, however, has not been devoted to second order stochastic optimization methods  and one would like to see how these algorithms fare here.\\

\item \textit{ODE computational methods } : The Euclidean Robbins-Monroe scheme can be thought of as a noisy discretization of the limiting ODE. In fact the SA scheme differs from the standard Euler scheme only in using a decaying time step and the presence of noise. The computational methods for solving ODEs on manifolds offer a lot more diversity (see Chapter 4, \cite{Hairer}). In fact the retraction scheme itself seems to draw a parallel with the projection methods used to solve ODEs on manifolds. One could perhaps adapt the various techniques there to the present case and study the properties of the resultant algorithms.   \\

\end{enumerate}

\end{document}